\documentclass[a4,12pt]{article}

\usepackage{amsfonts}
\usepackage{rotating}
\usepackage{chapterbib}
\usepackage{graphicx, color}
\usepackage[latin1]{inputenc}
\usepackage[english]{babel}
\usepackage{amssymb}
\usepackage{amsmath, eurosym, layout, stmaryrd, hyperref}
\usepackage{fancyhdr}
\usepackage{times}
\usepackage{amsthm}

\theoremstyle{definition}

\newtheorem{theorem}{Theorem}[section]

\newtheorem{definition}[theorem]{Definition}

\newtheorem{remark}[theorem]{Remark}

\newcommand{\N}{\mathbb{N}}

\renewcommand{\P}{\mathcal{P}}

\newcommand{\A}{\mathfrak{A}}

\newcommand{\G}{\mathfrak{G}}
\newcommand{\W}{\mathfrak{W}}

\newcommand{\B}{\mathfrak{B}}

\newcommand{\M}{\mathfrak{M}}

\renewcommand{\baselinestretch}{1.25}

\title{Power indices and minimal winning coalitions}

\author{Werner Kirsch and Jessica Langner}

\begin{document}
\maketitle

\renewcommand{\baselinestretch}{1.0}
\begin{abstract}
The \emph{Penrose-Banzhaf index} and the \emph{Shapley-Shubik index} are
the best-known and the most used tools to measure political power of voters in simple voting games. Most methods to calculate these power indices are based on counting winning coalitions, in particular those coalitions a voter is decisive for. We present a new combinatorial formula how to calculate both indices solely using the set of minimal winning coalitions.
\end{abstract}
\renewcommand{\baselinestretch}{1.25}


\section{Introduction}\label{sec:Introduction}
\thispagestyle{empty}
The theory of power indices is a systematic approach to measure political power in voting systems  (cp. \cite{Taylor1995}, \cite{FeMa1998}). Voting systems are also known as \emph{simple (voting) games} in literature.  The well-known \emph{Penrose-Banzhaf index} \cite{Penrose1946}, \cite{Banzhaf1965} and
\emph{Shapley-Shubik index} \cite{ShSh1954} rely on the concept of decisiveness of voters. On the other hand, the \emph{Deegan-Packel} index \cite{DePa1978} and the \emph{Holler-Packel index} \cite{HoPa1983} are based explicitly on the set of \emph{minimal winning coalitions} ($MWC$s). $MWC$s are those coalitions each voter is decisive for. Particularly, a calculation of power indices is easy to handle in \emph{weighted voting systems}. Here, voting weights are assigned to each voter and a decision threshold is defined. A proposal is accepted if the sum of the voting weights in favor meets or exceeds the given threshold.

Usually, calculation methods are based on listing the set of winning coalitions. In this paper we develop a combinatorial approach how to determine power indices solely using the $MWC$-set. For illustration we use the examples of the Penrose-Banzhaf index and the Shapley-Shubik index. It is known that each voting system (whether it is weighted or not) has got a $MWC$-set and it is completely defined by it. More precisely, each set of voting rules can be quantified by a $MWC$-set. Thus, our approach makes it possible to calculate power indices for each potential $MWC$-set in a rather elegant way. Furthermore, we could systematically calculate each potential constellation of voting power for a given set of voters. This might be useful for an optimization of existing voting systems or to design scientifically based proposals for further voting bodies.

\vspace{5mm}
This paper is organized as follows. In first part we present basic definitions and concepts of the theory of voting power. This section \ref{sec:bdac} is divided in three subsections. In subsection \ref{subsec:vs} voting systems will be defined. The theory of influence on decision in a voting body will be introduced in subsection \ref{subsec:damwc}. In this context minimal winning coalitions and its several properties will be discussed. In the last subsection \ref{subsec:pi} we present the best-known methods for measuring the mentioned decisiveness of voters. The section \ref{sec:Calc} is the main part of this paper. Here, we introduce our approach of a combinatorial calculation of the presented power indices solely using the set of minimal winning coalitions. How these calculation methods work will be illustrated using the example of the European Economic Community of 1958-1972 in section \ref{sec:ex}. The last section \ref{sec:Concl} of this paper contains concluding remarks.
\vspace{5mm}

\section{Basic definitions and concepts}\label{sec:bdac}
\subsection{Voting systems}\label{subsec:vs}
Voting systems consist of a set of voters and voting rules. The voting rules determine wether a proposal is accepted or not.

The \emph{set of voters} of a voting system can be represented by a finite non-empty set $W=\{1,\dots,n\}$. We call each element $w\in W$ a \emph{voter}. A collection of voters represented by a subset $A\subseteq W$ is called a \emph{coalition}. Its \emph{cardinality} $\#A$ is given by the number of voters in the coalition. The set of all coalitions of $W$ is denoted by $\P(W)$ which is called the \emph{power set} of $W$. Its cardinality is $\#\P(W)=2^n$. Additionally, we mention two important coalitions, the \emph{empty coalition} $\emptyset$ and the \emph{grand coalition} $W$.

Voters decide about accepting or rejecting a proposal by a vote in favor or against. Whenever we talk about a coalition we mean the collection of those voters who vote in favor of a given proposal.

Voting rules are reflected in a split of $\P(W)$ in two disjunct parts: The first part consists of those coalitions which can make a proposal pass; the second part consists of those coalitions which can not make a proposal pass. We call the first part the set of \emph{winning coalitions} and it will be denoted with $\G\subset\P(W)$. A coalition $A\not\in\G$ is called a \emph{losing coalition}. We will always assume that the grand coalition is a winning coalition while the empty coalition is a losing one. Moreover, we assume if $A$ is a winning coalition and the coalition $B$ comprises $A$, then $B$ should be winning as well. This property of $\G$ is called \emph{monotonicity}.

\begin{definition}
If $W$ is a finite non-empty set of voters and $\G$ is a monotone subset of $\P(W)$ with $W\in\G$ and $\emptyset\not\in\G$ a \emph{voting system} is a pair $\W:=(W,\G)$.
\end{definition}

In many applications it is obvious that either a coalition $A$ or its complementary coalition $(W\backslash A)$ is losing (or both are losing). A voting system with this property is called a \emph{proper} voting system. Otherwise, it is called \emph{improper}  (see e.g. \cite{FeMa1995}, \cite{FeMa1998}, \cite{FeMa1998b}). In the following we don't have to distinguish between proper and improper voting systems as our results are valid in both situations.

Frequently, voting systems consist of voting rules which assign voting weights to each voter and define a decision threshold. A proposal will be passed if the sum of the weights of the voters, which vote in favor, meets or exceeds the given threshold. These are the so-called \emph{weighted voting systems}.

\begin{definition}
A voting system $\W=(W,\G)$ is said to be \emph{weighted} if a function $g:W\rightarrow[0,\infty)$ and a number $q\in[0,\infty)$ exist with
\begin{eqnarray}
\sum_{w\in A}g(w)\geq q \textnormal{ holds for all }A\in\G\ \textnormal{ and } \sum_{w\in B}g(w)< q \textnormal{ holds for all }B\in(\P(W)\backslash\G)\ .
\end{eqnarray}
$g(w)$ is called the \emph{voting weight of $w$} and $q$ is called the \emph{quota}.
\end{definition}

\subsection{Decisiveness and minimal winning coalitions}\label{subsec:damwc}

An important aspect of political sciences is political power which is also known as \emph{voting power}. Voting power is a mathematical concept which quantifies the influence of a voter on election at a system. Its theory can be traced back to works of Penrose \cite{Penrose1946}, Shapley and Shubik \cite{ShSh1954} and Banzhaf \cite{Banzhaf1965}. If a voter can turn the voting outcome by changing his or her voting behavior (from vote in favor to against or vice versa) then he or she has influence on the voting decision (cp. \cite{Taylor1995}, \cite{FeMa1998} and \cite{Kirsch2004}). This property is known as \emph{decisiveness}. Thus, in a voting system $\W$ a voter $w$ is \emph{decisive} for a coalition $A\in\G$ if $w\in A$ and $(A\backslash \{w\})\not\in\G$. Otherwise, $w$ is said to be \emph{non decisive} for $A$.

Particularly, we consider those winning coalitions each voter in the coalition is decisive: A winning coalition $V\in\G$ is said to be a \emph{minimal winning coalition} $(MWC)$ if $V\backslash\{i\}$ is a losing coalition for each voter $i\in V$.

\begin{definition}\label{def:MWC}
The non-empty subset $\M(\G)$ with
\begin{eqnarray}
\M(\G):=\{V\in\G\mid V\textnormal{ is a }MWC\}
\end{eqnarray}
is called the $MWC$-\emph{set of} $\G$ (respectively $\W$).
\end{definition}

$MWC$-sets have various properties well known in set theory,
combinatorics and discrete mathematics: $MWC$-sets are
\emph{antichains} in $\P(W)$ which are also known as \emph{Sperner
families} (cp. \cite{Anderson1987} and \cite{Engel1997}) in
literature. More precisely, an antichain $\widetilde{\M}$ is a non-empty set of subsets of $W$ such that $X\nsubseteq Y$ and $Y\nsubseteq X$ holds for all $X,Y\in\widetilde{\M}$.

In addition, we observe that each voting system has a unique $MWC$-set due to monotonicity. In \cite{FeMa1998} the authors Felsenthal and Machover remarked that $\W$ is uniquely determined by its assembly $W$ and its $MWC$-set $\M(\G)$. Thus
\begin{eqnarray}
\G=\{A\in\P(W)\mid\exists\ V\in\M(\G):\ V\subseteq A\}.
\end{eqnarray}
This is due to the fact that minimal winning coalitions are just the minimal elements in $\G$ with respect to the partial order $\subseteq$. Also, each set $\G$ meets the conditions of an \emph{upset} or \emph{filter} (cp. \cite{Anderson1987} and
\cite{Engel1997}). Thus, a voting system is completely defined by its $MWC$-set. We call a $MWC$-set a \emph{basis} as well.

By a theorem of Sperner \cite{Sperner1928} on the cardinality of $\M(\G)$ it is known to satisfy:
\begin{eqnarray}
1\leq\#\M(\G)\leq{n\choose{\lfloor\frac{n}{2}\rfloor}}.
\end{eqnarray}

Furthermore, the number of different voting systems for a given number of voters is equal to the corresponding Dedekind number \cite{Dedekind1897} minus 2. According to the definition of $\G$ the two sets $\emptyset$ and $\{\emptyset\}$ are not allowed as $MWC$-set. The number of voting systems with up to eight voters is shown in table 1.

\begin{table}[h]
\begin{center} \label{tab:1}\caption[Number of antichains for a given set of $\#W$ voters.]{Number of antichains for a given set of $\#W$ voters.}
\begin{tabular}{|c||r|}
  \hline
  $\#W$ & Number of antichains\\ \hline\hline
     1 & 1\\\hline
     2 & 4\\\hline
     3 &  18\\\hline
     4 & 166\\\hline
     5 & 7.579\\\hline
     6 & 7.828.352\\\hline
     7 & 2.414.682.040.996\\\hline
     8 & 56.130.437.228.687.557.907.786\\\hline
\end{tabular}
\end{center}
\end{table}

\clearpage
\subsection{Power indices}\label{subsec:pi}
Voting power of each voter can be measured by power indices in terms of influence on decisions \cite{FeMa1998}. Felsenthal and Machover gave a general axiomatic definition of power indices in their papers \cite{FeMa1995} and \cite{FeMa1998b}. In the following we survey the two most popular power indices.
\begin{definition}[Penrose-Banzhaf index]
  \begin{eqnarray}
BS_w:=\#\{C\in\G\mid w\in C,\ (C\backslash\{w\})\not\in\G\}
\end{eqnarray}
is called the \emph{Banzhaf score} of a voter $w$ and
\begin{eqnarray}
PBP_w:=\frac{BS_w}{2^{n-1}}
\end{eqnarray}
is called the \emph{Penrose-Banzhaf power} of $w$. Finally
\begin{eqnarray}\label{def:BI}
PBI_w:=\frac{BS_w}{\sum_{i=1}^nBS_i}
\end{eqnarray}
is called the \emph{Penrose-Banzhaf index} of $w$.
\end{definition}

It is easy to see that $0 \leq PBI_w\leq 1$ and $\sum_{i=1}^n PBI_i=1$. The Penrose-Banzhaf power is equal to the probability a voter is decisive for a coalition. The Penrose-Banzhaf index measures the \emph{a priori} voting power of a voter. This means that the decisiveness of a voter will be measured without any previous knowledge of the single voters. Therefore it is natural to assume that all coalitions are equally likely.

\begin{definition}[Shapley-Shubik index]
\begin{eqnarray}\label{def:SSI}
SSI_w:=\sum_{{S\in\G}\atop{\textnormal{ with $w$ is decisive for } S}}\frac{(n-\#S)!(\#S-1)!}{n!}
\end{eqnarray}
is called the \emph{Shapley-Shubik index} of a voter $w$.
\end{definition}

As above $0 \leq SSI_w\leq 1$ and $\sum_{i=1}^n SSI_i=1$. The Shapley-Shubik index represents the fraction of orderings of voters for which a voter is decisive.

Both the Penrose-Banzhaf index and the Shapley-Shubik index measure the influence of voters in different ways (cp. \cite{Taylor1995}, \cite{FeMa1998} and \cite{LaVa2005}). In many cases they agree but important examples like the US federal system exist where they do not \cite{Taylor1995}. The right choice which index should be used for analysing a voting situation depends on the assumption about the voting behavior of the voters. In situations in which the voters vote completely independently from each other the Penrose-Banzhaf index should be used. Otherwise, if a common belief has influence on the choice of all voters the Shapley-Shubik index should be used (cp. \cite{Straffin1977}, \cite{LaVa2005} and \cite{Kirsch2007}).

Both the Penrose-Banzhaf index and the Shapley-Shubik are based on the decisiveness of voters. Moreover, further power indices exist which are based uniquely on the set of minimal winning coalitions, i.e. the Deegan-Packel index \cite{DePa1978} and the Holler-Packel index \cite{HoPa1983}.

In addition, the vector of the player's power values can be defined as \emph{power profile} concerning the power index under consideration.

We are interested in the set and values of potential power profiles of a given set of voters. It is a fact, that not every arbitrary constellation of voting power is possible (cp. \cite{Kirsch2001}). For example, in a voting system consisting of two voters only two power distributions are possible: Either one voter has got the total power and the other voter has no power \emph{or} both have the same (half) part of power. From section \ref{subsec:damwc} we know that the several $MWC$-sets of a given set of voters define each potential voting system. Thus, we are able to calculate each possible Deegan-Packel profile and Holler-Packel profile. The definitions of the Penrose-Banzhaf index and the Shapley-Shubik index allow us, in principle, to calculate these power indices by inspecting the list of all winning coalitions (cp. \cite{Taylor1995}, \cite{Leech2002} and \cite{Leech2003}). We have developed a new combinatorial approach to calculate the Penrose-Banzhaf index and the Shapley-Shubik index using simple terms which are solely based on the $MWC$-set.

\section{Calculations}\label{sec:Calc}

Firstly, we present a new calculation formula for the Banzhaf score of a voter. Out of this the Penrose-Banzhaf index can easily be determined. In the second part of this section we present a similar calculation method for the Shapley-Shubik index.

\begin{theorem}\label{th:tbnew}
[$BS$-direct-calculation formula] In a voting system $\W$ with $\M(\G)=\{V_1,\dots,V_m\}$ and $\#\M(\G)=m$ we have for each voter $w$
\begin{eqnarray}
BS_w=\sum_{r=1}^m(-1)^{r-1}\sum_{1\leq i_1<\dots<i_r\leq m}
t_{i_1,\dots,i_r}(w)
\end{eqnarray}
\begin{eqnarray} \textnormal{
with
}t_{i_1,\dots,i_r}(w):=\left\{\begin{array}{ll}2^{n-\#\bigcup_{j=1}^rV_{i_j}}
& ,\textnormal{ for }w\in\bigcup_{j=1}^rV_{i_j}, \\ 0 &
,\textnormal{ for } w\not\in\bigcup_{j=1}^rV_{i_j}.
\end{array}\right.
\end{eqnarray}
\end{theorem}

\begin{proof}
The value of $BS_w$ is equal to the cardinality of the subset
\begin{eqnarray}\label{def:A0w}
\A_{0w}:=\{A\in\G\mid w\in A,\ (A\backslash\{w\})\not\in\G\}.
\end{eqnarray}
This is the \emph{set of winning coalitions for which $w$ is decisive}. We can construct $\A_{0w}$ by using the following sets:
\begin{itemize}
  \item $\A_w:=\{A\in\G\mid w\in A\}$ is called the \emph{set of winning coalitions including $w$},
  \item $\A_{\not w}:=\{A\in\G\mid w\not\in A\}$ is called the \emph{set of winning coalitions excluding $w$},
  \item $\A_{1w}:=\{A\in\G\mid w\in A,\ (A\backslash\{w\})\in\G\}$ is called the \emph{set of winning coalitions for which $w$ is not decisive}.
\end{itemize}
$\#\A_{0w}$, $\#\A_w$, $\#\A_{\not w}$ and $\#\A_{1w}$ denote the cardinality of the respective set. \\ Obviously, $\A_w=\A_{0w}\cup\A_{1w}$ with $\A_{0w}\cap\A_{1_w}=\emptyset$, so $\A_{0w}=\A_w\backslash\A_{1w}$. Thus, we obtain
\begin{eqnarray}
\A_{0w}&=&\G\backslash(\A_{\not w}\cup\A_{1w})\quad\textnormal{with}\quad\A_{\not w}\cap\A_{1w}=\emptyset.
\end{eqnarray}
The sets $\G$ and $(\A_{\not w}\cup\A_{1w})$ can be constructed by
using \emph{principal filters}, i.e. for $A\in\P(W)$ $\B_A:=\{B\in\P(W)\mid
B\supseteq A\}$ is called the \emph{principal filter} of $A$. Let
$b_A:=\#\B_A$ denotes the cardinality of $\B_A$. Clearly, $b_A=2^{n-\#A}$. For $\#\M(\G)=m$ we obtain
\begin{eqnarray}\nonumber
\G&=&\{A\in \P(W)\mid\exists\ V\in\M(\G):\ V\subseteq
A\}\\&=&\bigcup_{i=1}^m\B_{V_i}
\end{eqnarray}
and
\begin{eqnarray}\nonumber
(\A_{\not w}\cup\A_{1w})&=&\left\{A\in\G\mid \left(w\not\in
A\right)\vee\left(\left(w\in
A\right)\wedge\left((A\backslash\{w\})\in\G\right)\right)\right\}\\\nonumber&=&\{A\in\P(W)\mid\exists
V\in\M(\G):V\subseteq A,\ w\not\in V\}\\&=&\bigcup_{i=1}^m\B'_{V_i}
\end{eqnarray}
with
\begin{eqnarray}
\B'_{V_i}:=\left\{\begin{array}{cl}\B_{V_i} & ,\ w\not\in
V_i,\\\emptyset&,\ w\in V_i.\end{array}\right.
\end{eqnarray}
Thus (cp. \cite{DuSh1979}, \cite{FeMa1998}),
\begin{eqnarray}
\A_{0w}&=&\G\backslash(\A_{\not
w}\cup\A_{1w})\\&=&\bigcup_{i=1}^m\B_{V_i}\backslash\bigcup_{i=1}^m\B'_{V_i}.
\end{eqnarray}

We can obtain the cardinality $\#\A_{0w}$ via the \emph{inclusion-exclusion principle} \cite{Steger2001}: For finite sets $A_1,\dots,A_n$ we have
\begin{eqnarray}\label{th:inexprinc}
\#\bigcup_{i=1}^nA_i=\sum_{r=1}^n(-1)^{r-1}\sum_{1\leq
i_1<\dots<i_r\leq n}\#\bigcap_{j=1}^rA_{i_j}.
\end{eqnarray}
For example, for two finite sets $A\,,B$ we have
\begin{eqnarray*}
\#(A\cup B)=\#A+\#B-\#(A\cap B)
\end{eqnarray*}
and for three finite sets $A\,,B\,,C$ we have
\begin{eqnarray*}
\#(A\cup B\cup C)&=&\#A+\#B+\#C\\&&-\left(\#(A\cap B)+\#(A\cap C)+\#(B\cap C)\right)\\&&+\#(A\cap B\cap C).
\end{eqnarray*}

To continue we need the cardinality of $\bigcap_{i=1}^m\B_{V_i}$. We have
\begin{eqnarray}\nonumber
\bigcap_{i=1}^m\B_{V_i}&=&\bigcap_{i=1}^m\{B\in\P(W)\mid B\supseteq V_i\}\\ &=& \{B\in\P(W)\mid B\supseteq \bigcup_{i=1}^m V_i\}\\ &=& \B_{\bigcup_{i=1}^mV_i}
\end{eqnarray}
and
\begin{eqnarray}\nonumber
\#\B_{\bigcup_{i=1}^mV_i}=2^{n-\#\bigcup_{i=1}^mV_i}.
\end{eqnarray}

Hence,
\begin{eqnarray}\nonumber
\#\bigcup_{i=1}^m\B_{V_i}&=&\sum_{r=1}^m(-1)^{r-1}\sum_{1\leq
i_1<\dots<i_r\leq m}\#\bigcap_{j=1}^r\B_{V_{i_j}}\\
&=&\sum_{r=1}^m(-1)^{r-1}\sum_{1\leq
i_1<\dots<i_r\leq m}2^{n-\#\bigcup_{j=1}^rV_{i_j}}.
\end{eqnarray}

We conclude
\begin{eqnarray}\nonumber
BS_w&=&\#\A_{0w}\\\nonumber
&=&\#\left(\bigcup_{i=1}^m\B_{V_i}\backslash\bigcup_{i=1}^m\B'_{V_i}\right)\\\nonumber
&=&\#\bigcup_{i=1}^m\B_{V_i}-\#\bigcup_{i=1}^m\B'_{V_i}\\\nonumber
&=&\left(\sum_{r=1}^m(-1)^{r-1}\sum_{1\leq i_1<\dots<i_r\leq
m}\#\bigcap_{j=1}^r\B_{V_{i_j}}\right)-\left(\sum_{r=1}^m(-1)^{r-1}\sum_{1\leq
i_1<\dots<i_r\leq m}\#\bigcap_{j=1}^r\B'_{V_{i_j}}\right)\\\nonumber
&=&\left(\sum_{r=1}^m(-1)^{r-1}\sum_{1\leq i_1<\dots<i_r\leq
m}\#\B_{\bigcup_{j=1}^r{V_{i_j}}}\right)-\left(\sum_{r=1}^m(-1)^{r-1}\sum_{1\leq
i_1<\dots<i_r\leq m}\#\B'_{\bigcup_{j=1}^r{V_{i_j}}}\right)\\&=&
\sum_{r=1}^m(-1)^{r-1}\sum_{1\leq i_1<\dots<i_r\leq
m}\underbrace{\left(\#\B_{\bigcup_{j=1}^rV_{i_j}}-{\#\B'_{\bigcup_{j=1}^rV_{i_j}}}\right)}_{=t_{i_1,\dots,i_r}(w)}.
\end{eqnarray}
To compute the term $t_{i_1,\dots,i_r}(w)$ we have to distinguish the following cases:
\begin{enumerate}
    \item If $w\in V_{i_j}$ for some $j$, thus $w\in\bigcup_{j=1}^rV_{i_j}$, then $\bigcap_{j=1}^r\B'_{V_{i_j}}=\emptyset$ and $\#\B'_{\bigcup_{j=1}^rV_{i_j}}=0$ according to assumption.

    \item If $w\not\in V_{i_j}$ for any $j$, thus $w\not\in\bigcup_{j=1}^rV_{i_j}$, then $\bigcap_{j=1}^r\B'_{V_{i_j}}=\bigcap_{j=1}^r\B_{V_{i_j}}$.
\end{enumerate}
We obtain
\begin{eqnarray}
t_{i_1,\dots,i_r}(w)=\left\{\begin{array}{cl}2^{n-\#\bigcup_{j=1}^rV_{i_j}} &,\ w\in\bigcup_{j=1}^rV_{i_j},\\ 0&,\ w\not\in\bigcup_{j=1}^rV_{i_j}\end{array}\right.
\end{eqnarray}
which completes the proof of theorem \ref{th:tbnew}.
\end{proof}

In an analogous manner, we can calculate the Shapley-Shubik index with the combinatorial model as above. For comparison, the Banzhaf score is characterized by allocating the value 1 for each coalition for which a voter is decisive. On the other hand, the Shapley-Shubik index allocates the value $\psi_{SSI}(A)=\frac{(n-\#A)!(\#A-1)!}{n!}$ for each coalition $A$ for which a voter is decisive.

Hence, the Shapley-Shubik index can be calculated with
\begin{eqnarray}
SSI_w=\sum_{A\in\A_{0w}}\psi_{SSI}(A)
\end{eqnarray}
with $\A_{0w}$ as defined in (\ref{def:A0w}).

\begin{theorem}\label{th:ssinew} [$SSI$-direct-calculation formula]. In a voting system $\W$ with $\M(\G)=\{V_1,\dots,V_m\}$ and $\#\M(\G)=m$ for each voter $w$ we have
\begin{eqnarray}
SSI_w=\sum_{r=1}^m(-1)^{r-1}\sum_{1\leq i_1<\dots<i_r\leq m}
s_{i_1,\dots,i_r}(w)
\end{eqnarray}
\begin{eqnarray} \textnormal{
with
}s_{i_1,\dots,i_r}(w):=\left\{\begin{array}{ll}\frac{1}{\#\bigcup_{j=1}^rV_{i_j}}
& ,\textnormal{ for }w\in\bigcup_{j=1}^rV_{i_j}, \\ 0 &
,\textnormal{ for } w\not\in\bigcup_{j=1}^rV_{i_j}.
\end{array}\right.
\end{eqnarray}
\end{theorem}

\begin{proof}
As above, we can construct $\A_{0w}$ from the $MWC$-set by use of principal filters. For this we can apply the inclusion-exclusion principle as well. (We get the proof by application of (\ref{th:inexprinc}) and complete induction.) We obtain for a coalition $A$ with $\#A=k$

\begin{eqnarray}\nonumber
\sum_{C\in\B_A}\psi_{SSI}(C)&=&\sum_{C\in\B_A}\frac{(n-\#C)!(\#C-1)!}{n!}\\\nonumber
&=&\sum_{\ell=k}^{n}{{n-k}\choose{\ell-k}}\frac{(n-\ell)!(\ell-1)!}{n!}\\\nonumber
&=&\frac{1}{k}\frac{1}{{n\choose k}}\sum_{\ell=k}^{n}\frac{(\ell-1)!}{(\ell-k)!(k-1)!}\\\label{rem:beforepastri}
&=&\frac{1}{k}\frac{1}{{n\choose k}}\sum_{\ell=k}^{n}{\ell-1\choose k-1}\\\label{rem:pascaltriangle}
&=&\frac{1}{k}\frac{1}{{n\choose k}}\left(1+\underbrace{\sum_{\ell=k+1}^{n}\left({\ell\choose k}-{\ell-1\choose k}\right)}_{\textnormal{telescope sum}}\right)\\\nonumber
&=&\frac{1}{k}\frac{1}{{n\choose k}}\left(1+{n\choose k}-1\right)\\\label{rem:1durchk}
&=&\frac{1}{k}.
\end{eqnarray}

For the step from (\ref{rem:beforepastri}) to (\ref{rem:pascaltriangle}) we used \emph{Pascal's Triangle} (cp. \cite{Steger2001}): For any $n,k\in\N$ with $n>k$ we gain ${n\choose k}={n-1\choose k-1} +{n-1\choose k}$. Hence,

\begin{eqnarray}\nonumber
SSI_w&=&\sum_{A\in\A_{0w}}\psi_{SSI}(A)\\\nonumber
&=&\sum_{A\in(\bigcup_{i=1}^m\B_{V_i}\backslash\bigcup_{i=1}^m\B'_{V_i})}\psi_{SSI}(A)\\\nonumber
&=&\sum_{A\in\bigcup_{i=1}^m\B_{V_i}}\psi_{SSI}(A)-\sum_{A\in\bigcup_{i=1}^m\B'_{V_i}}\psi_{SSI}(A)\\\nonumber
&=&\left(\sum_{r=1}^m(-1)^{r-1}\sum_{1\leq
i_1<\dots<i_r\leq m}\sum_{A\in\bigcap_{j=1}^r\B_{V_{i_j}}}\psi_{SSI}(A)\right)\\\nonumber&&-\left(\sum_{r=1}^m(-1)^{r-1}\sum_{1\leq
i_1<\dots<i_r\leq m}\sum_{A\in\bigcap_{j=1}^r\B'_{V_{i_j}}}\psi_{SSI}(A)\right)\\&=& \sum_{r=1}^m(-1)^{r-1}\sum_{1\leq
i_1<\dots<i_r\leq m}\left(\underbrace{\sum_{A\in\B_{\bigcup_{j=1}^rV_{i_j}}}\psi_{SSI}(A)-\sum_{A\in\B'_{\bigcup_{j=1}^rV_{i_j}}}\psi_{SSI}(A)}_{=s_{i_1,\dots,i_r}(w)}\right)
\end{eqnarray}

To compute the term $s_{i_1,\dots,i_r}(w)$ we have to distinguish the two cases as above. With (\ref{rem:1durchk}) we gain
\begin{eqnarray}
\sum_{A\in\B_{\bigcup_{j=1}^rV_{i_j}}}\psi_{SSI}(A)=\frac{1}{\#\bigcup_{j=1}^rV_{i_j}}
\end{eqnarray}

Finally, we obtain
\begin{eqnarray}
s_{i_1,\dots,i_r}(w)=\left\{\begin{array}{cl}\frac{1}{\#\bigcup_{j=1}^rV_{i_j}}&,\ w\in\bigcup_{j=1}^rV_{i_j},\\0&,\ w\not\in\bigcup_{j=1}^rV_{i_j}.\end{array}\right.
\end{eqnarray}
This completes the proof of theorem \ref{th:ssinew}.
\end{proof}

\section{Example}\label{sec:ex}
Using the simple example of the European Economic Community (EEC) of 1958-1972 we illustrate how these calculation methods work.
\begin{remark}
The EEC consisted of the six countries France (F), Germany (G), Italy (I), Belgium (B), Netherlands (N) and Luxembourg (L). Each country was assigned a voting weight as follows: (F,4), (G,4), (I,4), (B,2), (N,2) and (L,1). A proposal was accepted if the sum of the voting weights of the states voting in favor met or exceeded a quota of 12.
\end{remark}
This system can also be described by its four $MWC$s. We have the $MWC$-set
\begin{eqnarray}
\M_\textnormal{EEC}=\big\{\{F,G,I\},\ \{F,G,B,N\},\ \{F,I,B,N\},\ \{G,I,B,N\}\big\}.
\end{eqnarray}

We survey $\M_\textnormal{EEC}$ to calculate the Penrose-Banzhaf index and the Shapley-Shubik index using the example of France. To determine the values of the respective indices we have to consider each potential union of $MWC$s. These are a total of $2^4-1=15$ coalitions. If France takes part of a union of $MWC$s we either have to add or to subtract the value of the term $2^{n-\#\bigcup_{j=1}^rV_{i_j}}$ in the case of Penrose-Banzhaf or $\frac{1}{\#\bigcup_{j=1}^rV_{i_j}}$ in the case of Shapley-Shubik. Here, $\bigcup_{j=1}^rV_{i_j}$ stands for the union of some $V_{i_j}\in\M_\textnormal{EEC}$ and $\#\bigcup_{j=1}^rV_{i_j}$ stands for its cardinality. For example,  $\#\bigcup_{1}^1\{F,G,I\}=3$ and $\#\big\{\{F,G,I\}\cup\{F,G,B,N\}\big\}=\#\{F,G,I,B,N\}=5$. On the other hand, if France does not take part of a union we have to add or subtract zero.

The calculation consists of four steps because we have four $MWC$s. First, we consider each $MWC$ on its own and sum up the particular values of the respective term. The terms $R$ and $S$ stand for temporary remainders.
\begin{description}
  \item[Step one $BS_F$.]
  \begin{eqnarray}
  BS_F=2^{6-3}+2\cdot2^{6-4}+R_1=16+R_1.
  \end{eqnarray}
  \item[Step one $SSI_F$.]
  \begin{eqnarray}
  SSI_F=\frac{1}{3}+2\cdot\frac{1}{4}+S_1=\frac{5}{6}+S_1.
  \end{eqnarray}
\end{description}

On second step we consider each union of exact two $MWC$s. We subtract the particular values from the results of step one. Thus,
\begin{description}
  \item[Step two $BS_F$.]
  \begin{eqnarray}
  BS_F=16-\big(6\cdot2^{6-5}\big)+R_2=4+R_2.
  \end{eqnarray}
  \item[Step two $SSI_F$.]
  \begin{eqnarray}
  SSI_F=\frac{5}{6}-\big(6\cdot\frac{1}{5}\big)+S_2=-\frac{11}{30}+S_2.
  \end{eqnarray}
\end{description}

After this, on third step we consider each union of exact three $MWC$s. The respective values must be added to the results of step two. We obtain

\begin{description}
  \item[Step three $BS_F$.]
  \begin{eqnarray}
  BS_F=4+\big(4\cdot2^{6-5}\big)+R_3=12+R_3.
  \end{eqnarray}
  \item[Step three $SSI_F$.]
  \begin{eqnarray}
  SSI_F=-\frac{11}{30}+\big(4\cdot\frac{1}{5}\big)+S_3=\frac{13}{30}+S_3.
  \end{eqnarray}
\end{description}

On last step we consider the union of all four $MWC$s. The respective value must be subtracted from the result of step three.

\begin{description}
  \item[Final step $BS_F$.]
  \begin{eqnarray}
  BS_F=12-2^{6-5}=10.
  \end{eqnarray}
  \item[Final step $SSI_F$.]
  \begin{eqnarray}
  SSI_F=\frac{13}{30}-\frac{1}{5}=\frac{7}{30}.
  \end{eqnarray}
\end{description}

To calculate the Penrose-Banzhaf index of France we have to determine the Banzhaf scores of the remaining countries. Thus, we obtain $BS_F=BS_G=BS_I=10$, $BS_B=BS_N=6$ and $BS_L=0$. Hence, the final results are
\begin{eqnarray}
PBI_F=\frac{5}{21} \textnormal{ and } SSI_F=\frac{7}{30}.
\end{eqnarray}

\section{Conclusions}\label{sec:Concl}

In this paper we presented a new approach for the calculation of the two most popular power indices, the Penrose-Banzhaf index and the Shapley-Shubik index. For this we developed a combinatorial calculation method solely based on the set of minimal winning coalitions. This approach can be used to calculate the distribution of voting power in any arbitrary voting system in an easy way. It does not require several sets of voting rules to create each potential voting systems. Additionally, this method might also be used as model for calculating other power indices.

In comparison with existing calculation methods which are based on listing winning coalitions we predict that our approach might require more time for calculations in more complex voting system. The main purpose of the presented work is not to improve the existing calculation methods. Rather, it can be used for a systematically determining of several power profiles of a given set of voters. One aim might be to determine the complete set of potential power profiles. This might be useful for optimizing existing systems or for designing new voting bodies.



\vspace{5mm}
\emph{Werner Kirsch.}\\Fakultät für Mathematik und Informatik, FernUniversität Hagen, D-58095 Hagen, Germany.\\ \emph{Email:} werner.kirsch@fernuni-hagen.de\vspace{2.5mm}\\
\emph{Jessica Langner.}\\ Fakultät für Mathematik, Ruhr-Universität Bochum, D-44780 Bochum, Germany.\\ \emph{Email:} jessica.langner@ruhr-uni-bochum.de
\label{LastPage}
\clearpage
\end{document}